\documentclass{amsart}
\usepackage{amsfonts,amssymb,amscd,amsmath,mathabx,enumerate,verbatim,calc}
\usepackage{amscd,amssymb,amsopn,amsmath,amsthm,graphics,amsfonts,enumerate,verbatim,calc}
\usepackage[dvips]{graphicx}
\usepackage[colorlinks=true,linkcolor=red,citecolor=blue]{hyperref}
\input xy
\xyoption{all}
\calclayout

\newcommand{\C}{\mathbb{C} }
\newcommand{\D}{\mathbb{D} }
\newcommand{\K}{\mathbb{K} }

\newcommand{\Z}{\mathbb{Z} }
\newcommand{\TT}{\mathbb{T} }
\newcommand{\Tt}{\mathbf{T} }

\newcommand{\CQ}{\mathcal{Q} }

\newcommand{\Y}{\mathbf{Y}}

\newcommand{\T}{\bf{T}}

\newcommand{\Mod}{{\rm{Mod\mbox{-}}}}

\newcommand{\mmod}{{\rm{{mod\mbox{-}}}}}

\newcommand{\Inj}{{\rm{Inj}\mbox{-}}}
\newcommand{\Injj}{{\rm{inj}\mbox{-}}}
\newcommand{\Proj}{{\rm{Prj}\mbox{-}}}
\newcommand{\Projj}{{\rm{prj}\mbox{-}}}

\newcommand{\QR}{{{\rm Rep}{(\CQ, R)}}}

\newcommand{\bb}{{\rm{b}}}
\newcommand{\im}{{\rm{Im}}}

\newcommand{\sg}{{\rm{sg}}}

\newcommand{\ac}{{\rm{ac}}}
\newcommand{\tac}{{\rm{tac}}}

\newcommand{\HE}{{\rm{H}}}

\newcommand{\ZE}{{\rm{Z}}}
\newcommand{\BE}{{\rm{B}}}
\newcommand{\CK}{{\rm{C}}}
\newcommand{\Coker}{{\rm{Coker}}}
\newcommand{\Ker}{{\rm{Ker}}}

\newcommand{\HH}{{\rm{H}}}

\newcommand{\Hom}{{\rm{Hom}}}

\newcommand{\Thick}{{\rm{Thick}}}

\newtheorem{theorem}{Theorem}[section]
\newtheorem{corollary}[theorem]{Corollary}
\newtheorem{lemma}[theorem]{Lemma}

\newtheorem{proposition}[theorem]{Proposition}

\theoremstyle{definition}
\newtheorem{definition}[theorem]{Definition}
\newtheorem{example}[theorem]{Example}

\newtheorem{construction}[theorem]{Construction}
\newtheorem{remark}[theorem]{Remark}

\theoremstyle{plain}

\theoremstyle{definition}

\numberwithin{equation}{section}

\begin{document}

\title[Homotopy category of N-complexes of projective modules]{Homotopy category of N-complexes of projective modules}

\author[Bahiraei, Hafezi, Nematbakhsh]{P.Bahiraei, R.Hafezi, A.Nematbakhsh}

\address{Department of Mathematics, University of Isfahan, P.O.Box: 81746-73441, Isfahan, Iran and School of Mathematics, Institute for Research in Fundamental Science (IPM), P.O.Box: 19395-5746, Tehran, Iran }  \email{pbahiraei.math@sci.ui.ac.ir}
\email{hafezi@ipm.ir} \email{nematbakhsh@ipm.ir}



\begin{abstract}
In this paper, we show that the homotopy category of $N$-complexes of projective $R$-modules is triangle equivalent to the homotopy category of projective $\TT_{N-1}(R)$-modules where $\TT_{N-1}(R)$ is the ring of triangular matrices of order $N-1$ with entries in $R$. We also define the notions of $N$-singularity category and $N$-totally acyclic complexes. We show that the category of $N$-totally acyclic complexes of  finitely generated projective $R$-modules embeds in the $N$-singularity category, which is a result analogous to the case of ordinary chain complexes.

\end{abstract}

\maketitle

\section{introduction}
Given an associative unitary ring $R$, by an $N$-complex $X^\bullet$, we mean a sequence of $R$-modules and $R$-linear maps
$\cdots \rightarrow X^{n-1} \rightarrow X^n \rightarrow X^{n+1} \rightarrow \cdots $
such that composition of any $N$ consecutive maps gives the zero map.
The notion of $N$-complexes first appeared in the paper \cite{Kap96} by Kapranov. Besides their applications in theoretical physics \cite{CSW07}, \cite{Hen08}, the homological properties of $N$-complexes have become a subject of study for many authors as in \cite{DV98}, \cite{Est07}, \cite{Gil12}, \cite{GH10}, \cite{Tik02}.
Iyama and et. al. studied the homotopy category $\K_N(\mathcal B)$ of $N$-complexes of an additive category $\mathcal B$ as well as the derived category $\D_N(\mathcal A)$ of an abelian category $\mathcal A$. Recall that an abelian category $\mathcal{A}$ is an $(Ab 4)$-category(resp. $(Ab 4)^\ast$-category) provided that it has any coproduct(resp. product) of objects, and that the coproduct(resp. product) of monomorphisms(resp. epimorphisms) is monic(resp. epic). 
In the paper, \cite{IKM14}, they showed that the well known equivalences between homotopy category of chain complexes and their derived categories also generalize to the case of $N$-complexes. More precisely, if $\mathcal A$ is an abelian category satisfying the condition $(Ab4)$, then we have triangle equivalence
$$\K_N^\natural (\Proj \mathcal A) \cong \D_N^\sharp(\mathcal A).$$
where $(\natural,\sharp) = (proj, nothing), (-,-), ((-,b),b)$ and $\Proj \mathcal A$ is the category of projective objects of $\mathcal A$. As for chain complexes a similar statement is also true for the category $\Inj \mathcal A$ of injective objects of $\mathcal A$ provided that $\mathcal A$ satisfies the condition $(Ab 4)^\ast$.
They also showed that there exists a triangle equivalence
\begin{equation}
\label{eq-000}
\D_N(\mathcal A)\cong \D(\TT_{N-1}(\mathcal A)).
\end{equation}
As a consequence of this equivalence they showed that there exists the following triangle equivalences between derived and homotopy categories.
\begin{corollary} \label{cor-01}
For a ring $R$, we have the following triangle equivalences.
$$\D_N^\natural(\Mod R) \cong \D^\natural(\Mod \TT_{N-1}(R)),$$
where $\natural =$ $-,b$.
$$\K_N^\natural(\Proj R) \cong \K^\natural(\Proj \TT_{N-1}(R)),$$
where $\natural = -,b,(-,b)$ and also
$$\K_N^\natural(\Projj R) \cong \K^\natural(\Projj \TT_{N-1}(R)),$$
where $\natural = -,b,(-,b).$
\end{corollary}
In this paper, we show that the homotopy category $\K_N(\Proj(R))$ of $N$-complexes is embedded in the ordinary homotopy category $\K(\Proj \TT_{N-1}(R))$. Having this embedding in hand we are able recover (\ref{eq-000}) by using different techniques than those in \cite{IKM14}. We also show that $\K_N(\Proj(R))$ is equivalent to $\K(\Proj \TT_{N-1}(R))$ whenever $R$ is a left coherent ring. 

The explicit construction of such triangle equivalence allows us to prove an $N$-complex version of the following equivalence of triangulated categories given in \cite{Buc87}, \cite{Hap91}, \cite{BJO14}.
$$\K_{\tac}(\Projj R) \rightarrow \D_{\sg}^{\bb} (R)$$
where $\K_{\tac}(\Projj R)$ is the homotopy category of totally acyclic complexes of finitely generated projective $R$-modules and $\D_{\sg}^{\bb} (R)$ is the singularity category.

The paper is organized as follows. In section \ref{preliminaries}, we recall some generalities on $N$-complexes and provide any background information needed through this paper.
Our main result appears in section \ref{Some triangle equivalences between homotopy categories} as Theorem \ref{theorem 02}. In that section, we show that the category $\K_N(\Proj R)$ embeds as a triangulated subcategory in the category $\K(\Proj \TT_{N-1}(R))$ see proposition \ref{proposition 10}. As an application of this embedding we provide a different proof for the triangle equivalence in (\ref{eq-000}). At the end of this section we show that this embedding is also dense, hence an equivalence.

In section \ref{N-totally acyclic complexes} we define an $N$-totally acyclic complex as a complex $X^\bullet$ in $\Projj R$ satisfying the property that for all $P^\bullet \in \K^{\bb}_N(\Projj R)$, $\Hom_{K_{N}(\Projj R)} (P^\bullet,X^\bullet)= \Hom_{\K_{N}(\Projj R)} (X^\bullet, P^\bullet)=0$. Then we show that the homotopy category $\K_N^{\tac}(\Projj R)$ of $N$-totally acyclic complexes in $\Projj R$ in this sense is triangle equivalent to the homotopy category of ordinary totally acyclic complexes in $\Projj \TT_{N-1}(R)$, i.e. $\K_{\tac}(\Projj \TT_{N-1}(R))$. We also define a similar notion of singularity category for $N$-complexes $\D^{\sg}_N(R)$ and show that it contains $\K_N^{\tac}(\Projj R)$ as a triangulated subcategory. Furthermore, the embedding
$$\K_N^{\tac}(\Projj R) \rightarrow \D_N^{\sg}(R)$$
is an equivalence of triangulated categories, when $R$ is a Gorenstein ring.
\section{preliminaries}
\label{preliminaries}
\subsection{The category of $N$-complexes}
Throughout, $R$ is an associative ring with identity. $\Mod R$ denotes the category of all
right $R$-modules. We fix a positive integer $N\geq 2$. An $N$-complex $X^\bullet$ is a diagram
$$\xymatrix{\cdots \ar[r]^{{d}^{i-1}_{X^\bullet}} & X^{i} \ar[r]^{{d}^{i}_{X^\bullet}} &X^{i+1} \ar[r]^{{d}^{i+1}_{X^\bullet}}  & \cdots }$$ 
with $X^i \in \Mod R$ and morphisms $d^{i}_{X^\bullet} \in \Hom_{R}(X^i,X^{i+1})$ satisfying $d^N=0$. That is, composition of any $N$-consecutive maps is 0. A morphism between $N$-complexes is a commutative diagram
$$\xymatrix{\cdots \ar[r]^{{d}^{i-1}_{X^\bullet}} & X^{i} \ar[r]^{{d}^{i}_{X^\bullet}} \ar[d]^{f^i} & X^{i+1} \ar[r]^{{d}^{i+1}_{X^\bullet}} \ar[d]^{f^{i+1}}  & \cdots \\ \cdots \ar[r]^{{d}^{i-1}_{Y^\bullet}} & Y^{i} \ar[r]^{{d}^{i}_{Y^\bullet}}& Y^{i+1} \ar[r]^{{d}^{i+1}_{Y^\bullet}} & \cdots  }$$
We denote by $\C_N(R)$(resp. $\C^{-}_{N}(R)$, $\C^{+}_{N}(R)$, $\C^{b}_{N}(R)$) the category of unbounded (resp. bounded above, bounded below, bounded) $N$-complexes over $\Mod R$.
\\
For any object $M$ of $\Mod R$, $j\in \Z$ and $1\leq i \leq N$, let
$$\xymatrixcolsep{2.20pc}\xymatrix{  D^{j}_{i}(M): \cdots \ar[r] & 0 \ar[r] & X^{j-i+1} \ar[r]^{\,\,\,\,{d}^{j-i+1}_{X^\bullet}} & \cdots \ar[r]^{{d}^{j-2}_{X^\bullet}} & X^{j-1} \ar[r]^{{d}^{j-1}_{X^\bullet}}  & X^j \ar[r] &0 \ar[r] & \cdots }$$
be an $N$-complex satisfying $X^n=M$ for all $j-i+1\leq n \leq j$ and $d^{n}_{X^\bullet}=1_M$ for all $j-i+1\leq n \leq j-1$.
\\
For $0\leq r<N$ and $i\in \mathbb{Z}$, we define
$$ d_{X^\bullet,\lbrace r\rbrace}^{i}:=d_{X^\bullet}^{i+r-1} \cdots d_{X^\bullet}^{i} $$
In this notation $d_{X^\bullet,\lbrace 1\rbrace}^i=d^i_{X^\bullet}$ and $d_{X^\bullet,\lbrace 0\rbrace}^i=1_{X_i}$.
\begin{definition}
Let $f:X^\bullet\longrightarrow Y^\bullet$ be a morphism in $\C_N(R)$. The mapping cone $C(f)$ of $f$ is defined as follows
$$C(f)^m=Y^m \oplus {\coprod}_{i=m+1}^{m+N-1}X^i, \quad d^{m}_{C(f)}=
\left[ \begin{array}{cccccc}

d_{Y^\bullet}^{m} & f^{m+1} & 0 & 0  &\cdots & 0 \\
0 & 0 & 1 & \ddots & \ddots & \vdots \\
  \vdots & \vdots & \ddots & \ddots & \ddots & 0 \\
   & 0 & \cdots &  & 1 & 0 \\
    0 & 0 & \cdots & \cdots & 0 & 1 \\
     0 & -d_{X^\bullet ,\lbrace N-1\rbrace}^{m+1} & -d_{X^\bullet ,\lbrace N-2\rbrace}^{m+2} & \cdots & \cdots & -d_{X^\bullet}^{m+N-1} \\
\end{array} \right]  $$
\end{definition}
Let $\mathcal{S}_N(R)$ be the collection of short exact sequences in $\C_N(R)$ of which each term is split short exact in $R$. Then it is easy to see that a category $(\C_N(R),\mathcal{S}_N(R))$ is an exact category such that for every $M\in \Mod R$ and every $i\in \Z$, $D_{N}^{-i+N-1}(M)$ is an $\mathcal{S}_N$-projective and $\mathcal{S}_N$-injective object of this category. Hence this category is a Frobenius category, See \cite[Proposition 1.5]{IKM14}.
\begin{definition}
A morphism $f:X^\bullet \longrightarrow Y^\bullet$ of $N$-complexes is called null-homotopic if there exists $s^i \in \Hom_{R}(X^i,Y^{i-N+1})$ such that
\[ f^i=\sum_{j=0}^{N-1} d_{Y^\bullet, \lbrace N-1-j\rbrace}^{i-(N-1-j)}s^{i+j}d^{i}_{X^\bullet, \lbrace j\rbrace}\]
 We denote the homotopy category of unbounded $N$-complexes by $\K_N(R)$.
\end{definition}
\begin{definition}
For $X^\bullet=(X^i,d^i)\in \C_N(R)$, we define a shift functor $\Theta:\C_N(R)\rightarrow \C_N(R)$ by 
$$\Theta(X^\bullet)^i=X^{i+1}, \qquad \Theta(d)^i=d^{i+1} .$$ 
 We also define suspension functor $\Sigma:\K_N(R)\longrightarrow \K_N(R)$ as follows
$$(\Sigma X^\bullet)^m={\coprod}_{i=m+1}^{m+N-1}X^i, \qquad d^{m}_{\Sigma X^\bullet}=
\left[ \begin{array}{cccccc}

0 & 1 & 0 & 0  &\cdots & 0 \\
 & 0 & \ddots & \ddots & \ddots & \vdots \\
  \vdots & \vdots & \ddots & \ddots & \ddots & 0 \\
   &  &  & \ddots & \ddots & 0 \\
    0 & 0 & \cdots & \cdots & 0 & 1 \\
     -d_{\lbrace N-1\rbrace}^{m+1} & -d_{\lbrace N-2\rbrace}^{m+2} & \cdots & \cdots & \cdot & -d^{m+N-1} \\
\end{array} \right]     
$$
$$(\Sigma^{-1} X^\bullet)^m={\coprod}^{i=m-1}_{m-N+1}X^i, \qquad d^{m}_{\Sigma^{-1} X^\bullet}=
\left[ \begin{array}{cccccc}

-d^{m-1} & 1 & 0 & \cdots  &\cdots & 0 \\
-d_{\lbrace 2\rbrace}^{m-1} & 0 & 1 & \ddots & \ddots & \vdots \\
  \vdots & \vdots & \ddots & \ddots & \ddots & 0 \\
   &  &  & \ddots & \ddots & 0 \\
    -d_{\lbrace N-2\rbrace}^{m-1} & 0 & \cdots & \cdots & 0 & 1 \\
     -d_{\lbrace N-1\rbrace}^{m-1} & 0 & \cdots & \cdots & \cdot & 0 \\
\end{array} \right]     
$$
\end{definition}
It is known that $\K_N(R)$  together with this suspension functor is a triangulated category, see  \cite[Theorem 1.7]{IKM14}.
\\
Let $X^\bullet$,
$$\xymatrix{\cdots \ar[r]^{{d}^{i-1}_{X^\bullet}} & X^{i} \ar[r]^{{d}^{i}_{X^\bullet}} &X^{i+1} \ar[r]^{{d}^{i+1}_{X^\bullet}}  & \cdots }$$
be an $N$-complex of $R$-modules. We define
$$\ZE^{i}_{r}(X^\bullet):= \Ker d^{i}_{X^\bullet,\lbrace r \rbrace} \quad , \quad \BE^{i}_{r}(X^\bullet):= \im d^{i-r}_{X^\bullet, \lbrace r \rbrace}
$$
$$
\CK^{i}_{r}(X^\bullet):= \Coker d^{i-r}_{X^\bullet,\lbrace r \rbrace} \quad , \quad \HE^{i}_{r}(X^\bullet):= \ZE^{i}_{r}(X^\bullet)/ \BE^{i}_{N-r}(X^\bullet). 
$$
In each degree we have $N-1$ cycles and clearly $\ZE^{n}_{N}(X^\bullet)=X^n$. 
\begin{remark}
\label{remark 01}
For any $X^\bullet\in \C_N(R)$ if $\HE^i_1(X^\bullet)=0$ for any $i\in \mathbb{Z}$, then we have $\HE^i_r(X^\bullet)=0$ for any $i\in \mathbb{Z}$ and $0<r<N$.
\end{remark}
\begin{definition}
Let $X^\bullet \in \K_N(R)$. We say $X^\bullet$ is $N$-exact if $\HE^{i}_{r}(X^\bullet)=0$
for each $i \in \mathbb{Z}$ and all $r=1,2,...,N-1$. We denote the full subcategory of $\K_N(R)$ consisting of $N$-exact complexes by $\K_{N}^{\ac}(R)$ . 
\end{definition}
For a full subcategory $\mathcal{B}$ of $\Mod R$, we denote by $\K_N^{\natural,\bb}(\mathcal{B})$ the full subcategory of $\K_N^{\natural}(\mathcal{B})$ consisting of $N$-complexes $X^\bullet$ satisfying $\HH^i_r(X^\bullet)=0$ for almost all but finitely many $i$ and $r$, where $\natural= nothing,-,+$.
\begin{definition}
A morphism $f:X^\bullet \rightarrow Y^\bullet$ is called quasi-isomorphism if the induced morphism $\HE^{i}_{r}(f):\HE^{i}_{r}(X^\bullet)\rightarrow \HE^{i}_{r}(Y^\bullet)$ is an isomorphism for any $i$ and $1\leq r\leq N-1$, or equivalently if the mapping cone $C(f)$ belongs to $\K_{N}^{\ac}(R)$. The derived category $\D_N(R)$ of $N$-complexes is defined as the quotient category $\K_{N}(R)/ \K_{N}^{\ac}(R)$. 
\end{definition}
\subsection{Triangular matrix ring}
Let $\mathbb{M}_n(R)$ be the set of all $n\times n$ square matrices with coefficients in $R$ for $n\in \mathbb{N}$. $\mathbb{M}_n(R)$ is a ring with respect to the
usual matrix addition and multiplication. The identity of $\mathbb{M}_n(R)$ is the
matrix $E = diag(1, . . . , 1)\in \mathbb{M}_n(R)$ with 1 on the main diagonal and zeros
elsewhere. The subset
$$\mathbb{T}_n(R)=
\left[ \begin{array}{cccc}

R & 0 & \cdots & 0 \\
R & R & \cdots & 0 \\
  \vdots & \vdots & \ddots & \vdots \\
     R & R & \cdots & R \\
\end{array} \right]     
$$
of $\mathbb{M}_n(R)$ consisting of all triangular matrices $[a_{ij}]$ in $\mathbb{M}_n(R)$ with zeros
over the main diagonal is a subring of $\mathbb{M}_n(R)$. It is well known that if $Q$ is the quiver 
$$\xymatrix{ A_n = 1 \ar[r] & 2 \ar[r] &3 \ar[r]  & \cdots \ar[r]& n  }$$ then $RQ\cong \mathbb{T}_n(R)$, where $RQ$ is a path algebra of quiver $Q$. 
\\
Let $Q=(V,E)$ be a quiver. A representation of $Q$ by  a ring $R$ is a correspondence which associates an object $M_v$ to each vertex $v$ and a morphism $\varphi_a:M_{s(a)}\rightarrow M_{t(a)}$ to each arrow $a\in E$. Let $\mathcal{X}$ and  $\mathcal{Y}$ be two representations by left $R$-modules of the quiver $Q$. A morphism $f:\mathcal{X}\rightarrow \mathcal{Y}$ is a family of homomorphisms $f_v:\mathcal{X}_v\rightarrow \mathcal{Y}_v$ shch that $\mathcal{Y}_a \circ f_v=f_w \circ \mathcal{X}_a$ for any arrow $a:v\rightarrow w$. The representations of $Q$ by $R$-modules and $R$-homomorphisms form a category denoted by $\QR$.
\\
It is known that the category $\QR$ is equivalent to the category of modules over path algebra $RQ$ whenever $Q$ is finite quiver.
\\
Set $Q=A_n$. For $1\leq i\leq n$, let $e^i: \QR\rightarrow \Mod R$ be the evaluation functor defined by $e^i(\mathcal{X})=\mathcal{X}_i$, for any $\mathcal{X}\in \QR$. It is proved in \cite{EH99} that $e^i$ has a right adjoint $e^i_{\rho}:\Mod R\rightarrow \QR$, where $e^i_{\rho}(M)$ is the following representation 
$$\xymatrix{  M \ar[r] & M \ar[r] &\cdots \ar[r] & M \ar[r] & 0 \ar[r] & \cdots \ar[r]& 0  } $$
for the $R$-module $M$, where $M$ ends in $i$-th position with identity morphisms beforehand. Moreover, for $1\leq j\leq n$, it is shown that $e^j$ also admits a left adjoint $e^j_{\lambda}$, defined by $e^j_{\lambda}(M)$ as follows:
$$\xymatrix{  0 \ar[r] & 0 \ar[r] &\cdots \ar[r] & M \ar[r] & M \ar[r] & \cdots \ar[r]& M  } $$
for the $R$-module $M$, where $M$ starts in $j$-th position with identity morphisms afterward. It is proved in \cite{EER09} that any projective(resp. injective) representation $\mathcal{X}$ in $\QR$ is of the form $\bigoplus_{i=1}^n e^i_{\lambda}(P^i)$(resp. $\bigoplus_{i=1}^n e^i_{\rho}(I^i)$), where for any $1\leq i\leq n$, $P^i$(resp. $I^i$), is the cokernel(resp. kernel) of the split monomorphism $\mathcal{X}^{i-1}\rightarrow \mathcal{X}^i$(resp. epimorphism $\mathcal{X}^{i}\rightarrow \mathcal{X}^{i+1}$). Hence any projective object in $\Mod \mathbb{T}_n(R)$ is of the form 
$$P^1\rightarrow P^1\oplus P^2 \rightarrow P^1\oplus P^2 \oplus P^3 \rightarrow \cdots \rightarrow P^1\oplus P^2 \oplus \cdots \oplus P^n$$
and an injective object in $\Mod \mathbb{T}_n(R)$ is of the form 
$$I^1\oplus I^2 \oplus \cdots \oplus I^n\rightarrow I^1\oplus I^2 \oplus \cdots \oplus I^{n-1} \rightarrow \cdots \rightarrow I^1\oplus I^2 \rightarrow  I^1$$
\section{Some triangle equivalences between homotopy categories }
\label{Some triangle equivalences between homotopy categories}
In this section, we show that the homotopy category $\K_N(\Proj(R))$ of $N$-complexes is embedded in the ordinary homotopy category $\K(\Proj \TT_{N-1}(R))$. As a result of this embedding we show that there exists a triangle equivalence between derived category of $N$-complexes and ordinary derived category of complexes of $\Mod \TT_{N-1}(R)$. At the end of this section we show that $\K_N(\Proj(R))\cong \K(\Proj \TT_{N-1}(R))$.  
\\
Let $\mathcal{S}_N(R)$ be the collection of short exact sequence in $\C_N(R)$ of which each term is split exact then it is shown in \cite{IKM14} that $(\C_N(R),\mathcal{S}_N(R))$ is a Frobenius category. We need the following definition and lemma from \cite{Hap88}.
\begin{definition}
Let $(\mathcal{B},\mathcal{S})$ and $(\mathcal{B}',\mathcal{S}')$ be Frobenius categories. An additive
functor $F:\mathcal{B}\longrightarrow \mathcal{B}'$ is called exact if $\xymatrix{0 \ar[r] & F(X) \ar[r]^{F(u)} & F(Y) \ar[r]^{F(v)}  & F(Z) \ar[r] & 0 }$ is contained in $\mathcal{S}'$ whenever $\xymatrix{0 \ar[r] & X \ar[r]^{u} & Y \ar[r]^{v}  & Z \ar[r] & 0 }$ is contained in $\mathcal{S}$. 

\vspace{1mm}
If $F$ transforms $\mathcal{S}$-injectives
into $\mathcal{S}'$-injectives then $F$ induces a functor $\underline{F}:\underline{\mathcal{B}}\longrightarrow \underline{\mathcal{B}'}$. Denote by $T$ (resp. $T'$) the translation
functor on stable category $\underline{\mathcal{B}}$ (resp. $\underline{\mathcal{B}'}$).
\end{definition} 
\begin{lemma}
Let $F$ be an exact functor between Frobenius categories
$\mathcal{B}$ and $\mathcal{B}'$ such that $F$ transforms $\mathcal{S}$-injectives into $\mathcal{S}'$-injectives.
If there exists an invertible natural transformation
$\alpha : FT \longrightarrow T'F $ then $F$ is an exact functor of triangulated categories.
\end{lemma}
In order to show that $\K_N(\Proj(R))$ embeds in $\K(\Proj \TT_{N-1}(R))$, we explicitly construct the embedding functor.
\begin{construction}
\label{construction 01}
Define the functor $\mathbf{F}:\C_N(\Proj R)\longrightarrow \C(\Proj \TT_{N-1}(R))$ by the following rules.

\vspace{2mm}
\item[\textbf{On objects:}]
Let $(P^\bullet,d^\bullet)$ be an object in $\C_N(\Proj R)$. Define $i$-th term of $\mathbf{F}(P^\bullet)$ as follows

\vspace{1mm}
\begin{itemize}
\item
For $i=2r$, let $m=Nr$ and define $\mathbf{F}(P^\bullet)^i$  as the following projective representation of $A_{N-1}$: 
\end{itemize} 
\begin{small}
$$
\xymatrix{
 P^m \ar[r]  & P^m\oplus P^{m+1} \ar[r]  & \cdots \ar[r]  & P^m\oplus P^{m+1}\oplus \cdots \oplus P^{m+N-3} \ar[r]  & P^m\oplus \cdots \oplus P^{m+N-2}   
 }
 $$
\end{small} 

\vspace{1mm}
\begin{itemize}
\item
For $i=2r+1$, let $m=Nr$ and define $\mathbf{F}(P^\bullet)^i$  as the following projective representation of $A_{N-1}$:
\end{itemize}
\begin{tiny}
$$
\xymatrix{
P^{m+N-1} \ar[r]  & P^{m+N-1}\oplus P^{m+N}\ar[r]  & \cdots \ar[r] & P^{m+N-1}\oplus P^{m+N}\oplus \cdots \oplus P^{m+2N-4} \ar[r]  & P^{m+N-1}\oplus \cdots \oplus P^{m+2N-3}  
}
$$
\end{tiny}

\vspace{1mm}
For the definition of differential of $\mathbf{F}(P^\bullet)$, we consider the following two cases:
\begin{itemize}
\item[(i)] For $i=2r$, define $\mu^i:\mathbf{F}(P^\bullet)^i\rightarrow \mathbf{F}(P^\bullet)^{i+1}$ by $\mu^i=(\mu^i_j)_{1 \leq j \leq N-1}$ where
\end{itemize} 

\vspace{1mm}
$$\mu_{j}^{i}=
\left[ \begin{array}{ccccc}

d_{\lbrace N-1\rbrace}^m & d_{\lbrace N-2\rbrace}^{m+1}  & \cdots  & d_{\lbrace N-j\rbrace}^{m+j-1} \\
 0 & d_{\lbrace N-1\rbrace}^{m+1} & \cdots  & d_{\lbrace N-j+1\rbrace}^{m+j-1} \\
  0 & 0 &  &  &  \\
    \vdots & \vdots & \ddots  &\vdots  \\
     &  &    &  \\
     0 & 0 & \cdots  & d_{\lbrace N-1\rbrace}^{m+j-1} \\
\end{array} \right]     
$$
\begin{itemize}

\vspace{3mm}
\item[(ii)] For $i=2r+1$, define $\lambda^i:\mathbf{F}(P^\bullet)^i\rightarrow \mathbf{F}(P^\bullet)^{i+1}$ by $\lambda^i=(\lambda^i_j)_{1 \leq j \leq N-1}$ where
\end{itemize}

\vspace{1mm}
$$\lambda_{j}^{i}=
\left[ \begin{array}{cccccc}
d^{m+N-1} & -1 & 0 & \cdots & \cdots & 0 \\
 0 & d^{m+N} & -1 & 0 & \cdots & 0 \\
  0 & 0 & d^{m+N+1} & -1 & \cdots & 0  \\
  \vdots & \vdots & & \ddots  &\ddots & \vdots  \\
 \vdots & \vdots &    & & & -1 \\
     0 & 0 & \cdots  & \cdots & 0 & d^{m+N+j-2} \\
\end{array} \right]     
$$
\item[\textbf{On morphisms:}]Let $f^\bullet:Q^\bullet \longrightarrow P^\bullet $ be a morphism in $\C_N(\Proj R)$. We define $\mathbf{F}(f^\bullet)$ as follows:
\\
\begin{itemize}
\item[(i)] For $i=2r$, let $m=Nr$. Define $\varphi^i:\mathbf{F}(Q^\bullet)^i\longrightarrow \mathbf{F}(P^\bullet)^i$ by $\varphi^i=(\varphi^i_j)_{0\leq j \leq N-2}$ where $$\varphi^i_j=diag(f^m,...,f^{m+j})$$
\end{itemize}

\vspace{2mm}
\begin{itemize}
\item[(ii)] For $i=2r+1$,let $m=Nr$. Define $\varphi^i:\mathbf{F}(Q^\bullet)^i\longrightarrow \mathbf{F}(P^\bullet)^i$ by $\varphi^i=(\varphi^i_j)_{0\leq j \leq N-2}$ where 
$$
\varphi^i_j=diag(f^{m+N-1},...,f^{m+N-j}). $$
\end{itemize}
It is straightforward to show that this construction defines covariant functor from $\C_N(\Proj R)$ to $\C(\Proj \T_{N-1}(R))$.
\end{construction}
\begin{example}
Let $N=3$. Let $P^\bullet$ 
$$\xymatrix{P^\bullet=\cdots \ar[r]^{{d}^{-2}} & P^{-1} \ar[r]^{{d}^{-1}} & P^{0} \ar[r]^{{d}^{0}}  & P^{1} \ar[r]^{{d}^{1}} & P^{2} \ar[r]^{{d}^{2}} & P^{3} \ar[r]^{{d}^{3}} & \cdots }$$
be a $3$-complex in $\C_3(\Proj R)$.
The functor $\mathbf{F}$ maps $P^\bullet$ in to the following complex in $\C(\Proj \TT_2(R))$
\begin{small}
$$\xymatrixcolsep{4.5pc}\xymatrix{ P^{-1} \ar[r]^{d^{-1}} \ar[d]^{(1,0)} & P^0 \ar[r]^{d^1 d^0} \ar[d]^{(1,0)} & P^2 \ar[r]^{d^2} \ar[d]^{(1,0)} & P^3 \ar[r]^{d^4 d^3} \ar[d]^{(1,0)} & P^5  \ar[d]^{(1,0)}  
\\
 P^{-1}\oplus P^0 \ar[r]^{\begin{tiny} \left[\begin{array}{cc}
d^{-1} & -1  \\
 0 & d^{0}  \\
\end{array} \right] \end{tiny}} & P^{0}\oplus P^1 \ar[r]^{\begin{tiny} \left[ \begin{array}{cc}
d^1d^0 & d^1 \\
 0 & d^2d^1  \\
\end{array} \right] \end{tiny}} & P^{2}\oplus P^3 \ar[r]^{\begin{tiny} \left[ \begin{array}{cc}
d^{2} & -1  \\
 0 & d^{3}  \\
\end{array} \right] \end{tiny}} & P^{3}\oplus P^4 \ar[r]^{\begin{tiny} \left[ \begin{array}{cc}
d^4d^3 & d^4 \\
 0 & d^5d^4  \\
\end{array} \right] \end{tiny}} & P^{5}\oplus P^6
 } $$
 \end{small}
 Now consider the morphism $f^\bullet:(Q^\bullet,e^\bullet) \longrightarrow (P^\bullet,d^\bullet) $ in $\C_3(\Proj R)$ as follows:
$$\xymatrix{\cdots \ar[r]^{{e}^{-2}} & Q^{-1} \ar[r]^{{e}^{-1}} \ar[d]^{f^{-1}} & Q^{0} \ar[r]^{{e}^{0}} \ar[d]^{f^{0}} & Q^{1} \ar[r]^{{e}^{1}} \ar[d]^{f^{1}} & Q^{2} \ar[r]^{{e}^{2}} \ar[d]^{f^{2}} & Q^{3} \ar[r]^{{e}^{3}} \ar[d]^{f^{3}} & \cdots
\\ 
\cdots \ar[r]^{{d}^{-2}} & P^{-1} \ar[r]^{{d}^{-1}} & P^{0} \ar[r]^{{d}^{0}}  & P^{1} \ar[r]^{{d}^{1}} & P^{2} \ar[r]^{{d}^{2}} & P^{3} \ar[r]^{{d}^{3}} & \cdots }$$
The image of $f^\bullet$ under $\mathbf F$ is the following diagram in $\C(\Proj \TT_{2}(R))$:
\begin{tiny}
$$\xymatrix{ & Q^{-1} \ar@{.>}[rr] \ar[dl] \ar@{.>}[dd]^<<<<<<<{f^{-1}} & & Q^{0} \ar@{.>}[rr] \ar[dl] \ar@{.>}[dd]^<<<<<<<{f^{0}} & & Q^2 \ar[dl] \ar@{.>}[rr] \ar@{.>}[dd]^<<<<<<<{f^2} & & Q^{3} \ar[dl] \ar@{.>}[dd]^<<<<<<<{f^{3}}
\\ 
Q^{-1}\oplus Q^0 \ar[rr] \ar[dd]|-{\begin{tiny} \left[ \begin{array}{cc}
f^{-1} & 0  \\
 0 & f^{0}  \\
\end{array} \right] \end{tiny}} & & Q^{0}\oplus Q^1 \ar[rr] \ar[dd]|-{\begin{tiny} \left[ \begin{array}{cc}
f^{0} & 0  \\
 0 & f^{1}  \\
\end{array} \right] \end{tiny}} & & Q^{2}\oplus Q^3 \ar[dd]|-{\begin{tiny} \left[ \begin{array}{cc}
f^{2} & 0  \\
 0 & f^{3}  \\
\end{array} \right] \end{tiny}} \ar[rr] & & Q^{3}\oplus Q^4  \ar[dd]|-{\begin{tiny} \left[ \begin{array}{cc}
f^{3} & 0  \\
 0 & f^{4}  \\
\end{array} \right] \end{tiny}} & 
\\ 
 & P^{-1} \ar@{.>}[rr] \ar[dl]  & & P^{0} \ar@{.>}[rr] \ar[dl] & & P^{2} \ar[dl] \ar@{.>}[rr] & & P^{3}  \ar[dl] 
\\ 
P^{-1}\oplus P^0 \ar[rr]  & & P^{0}\oplus P^1 \ar[rr]  & & P^{2}\oplus P^3 \ar[rr] &  & P^{3}\oplus P^4  &  
}$$
\end{tiny}
\end{example}
\begin{lemma}
\label{lemma 12}
The functor $\mathbf F$, defined above, induces a functor from the category $\K_N(\Proj R)$ to the category $\K(\Proj \TT_{N-1}(R))$ which we denote it again by $\mathbf F$.
\end{lemma}
\begin{proof}
We show that if $f^\bullet:(Q^\bullet,e^\bullet) \longrightarrow (P^\bullet,d^\bullet)$ is a null homotopic map in $\C_N(\Proj R)$ then $\mathbf{F}(f^\bullet)$ is a null homotopic map in $\C(\Proj \TT_{N-1}(R))$.
Since $f^\bullet\sim 0^\bullet$, by definition there exists $s^m \in \Hom_{R}(Q^m,P^{m-N+1})$ such that
\[ f^m=\sum_{k=0}^{N-1} d_{ \lbrace N-1-k\rbrace}^{m-(N-1-k)}s^{m+k}e^{m}_{\lbrace k\rbrace}\] 
We want to construct $(t^i)_{i\in \mathbb{Z}}$, such that $(\mathbf{F}(f^\bullet))^i=\lambda^{i-1}_{P^\bullet}t^i+t^{i+1}\mu^i_{Q^\bullet}$(resp. $(\mathbf{F}(f^\bullet))^i=\mu^{i-1}_{P^\bullet}t^i+t^{i+1}\lambda^i_{Q^\bullet}$) when $i$ is even(resp. odd). We consider the following two cases:

\vspace{3mm}
(i)
Let $i=2r$ and $m=Nr$. We define
$t^i=(t^i_j)_{1\leq j \leq N-1}$
where
\begin{tiny}
 $$t_{j}^{i}=
\left[ \begin{array}{ccccc}
\sum_{k=0}^{N-2}d_{\lbrace N-2-k\rbrace}^{m-(N-1-k)}s^{m+k}e_{\lbrace k \rbrace}^{m} & \sum_{k=1}^{N-2}d_{\lbrace N-2-k\rbrace}^{m-(N-1-k)}s^{m+k}e_{\lbrace k-1 \rbrace}^{m+1}  & \cdots  & \sum_{k=j-1}^{N-2}d_{\lbrace N-2-k\rbrace}^{m-(N-1-k)}s^{m+k}e_{\lbrace k-j+1 \rbrace}^{m+j} \\
 & & & \\
 0 & \sum_{k=0}^{N-2}d_{\lbrace N-2-k\rbrace}^{m-(N-2-k)}s^{m+k+1}e_{\lbrace k \rbrace}^{m+1} & \cdots  & \sum_{k=l-2}^{N-2}d_{\lbrace N-2-k\rbrace}^{m-(N-2-k)}s^{m+k+1}e_{\lbrace k-j+2 \rbrace}^{m+j} \\
  0 & 0 &  &  &  \\
    \vdots & \vdots & \ddots  &\vdots  \\
     0 & 0 & \cdots  & \sum_{k=0}^{N-2}d_{\lbrace N-2-k\rbrace}^{m-(N-j-k)}s^{m+k+j-1}e_{\lbrace k \rbrace}^{m+j} \\
\end{array} \right]     
$$
\end{tiny}

\vspace{2mm}
(ii)
Let $i=2r+1$ and $m=Nr$. We define:
$$ t^{i}=(t^{i}_j)_{1\leq j \leq N-1} \qquad \text{where} \qquad t^{i+1}_j=
\left[ \begin{array}{ccccc}

s^{m+N-1} & 0  & \cdots & 0 \\
 0 & s^{m+N} & \cdots  & 0 \\
  0 & 0 &  &  &  \\
    \vdots & \vdots & \ddots  &\vdots  \\
     0 & 0 & \cdots  & s^{m+N+j-2} \\
\end{array} \right]$$
\end{proof}
\begin{lemma}
\label{lemma 13}
The functor $\mathbf F:\K_N(\Proj(R))\longrightarrow \K(\Proj \TT_{N-1}(R))$ is a fully faithful functor.
\end{lemma}
\begin{proof}
Let $P^\bullet$ and $Q^\bullet$ be two objects in $\C_N(\Proj R)$. We want to show that:
$$\Hom_{\K_N(\Proj R)}(Q^\bullet,P^\bullet)\cong \Hom_{\K(\Proj \TT_{N-1}(R))}(\mathbf{F}(Q^\bullet),\mathbf{F}(P^\bullet)) $$
Let $f^\bullet:(Q^\bullet,e^\bullet) \longrightarrow (P^\bullet,d^\bullet)$ be a morphism in $\C_N(\Proj R)$ such that $\mathbf{F}(f^\bullet)=0$ in $\K(\Proj \TT_{N-1}(R))$. We want to show that $f^\bullet=0$ in $\K_N(\Proj R)$. We construct $s^m:Q^m\rightarrow P^{m-N+1}$ such that
\begin{equation} 
f^m=\sum_{k=0}^{N-1} d_{ \lbrace N-1-k\rbrace}^{m-(N-1-k)}s^{m+k}e^{m}_{\lbrace k\rbrace} \label{eq-00}
\end{equation}
for all $m\in \Z$.
\\
Suppose $i=Nr$ for some $r\in \Z$. Consider the following diagram:
$$
\xymatrix{Q^{i} \ar[r]^{e^{i}} \ar[d]^{f^{i}} & Q^{i+1} \ar[r]^{e^{i+1}} \ar[d]^{f^{i+1}} & Q^{i+2} \ar[r]^{e^{i+2}} \ar[d]^{f^{i+2}} &  \cdots  \ar[r]  & Q^{i+N-2} \ar[r]^{e^{i+N-2}} \ar[d]^{f^{i+N-2}} & Q^{i+N-1} \ar[d]^{f^{i+N-1}}  \\ P^{i} \ar[r]^{d^{i}} & P^{i+1} \ar[r]^{d^{i+1}}& P^{i+2} \ar[r]^{d^{i+2}} &  \cdots \ar[r] &  P^{i+N-2} \ar[r]^{d^{i+N-2}}  & P^{i+N-1}  } 
$$
Since $\mathbf{F}(f^\bullet)\sim 0^\bullet$ there exists 
 
$$ t^{n}=(t^{n}_j)_{1\leq j \leq N-1} \qquad \text{where} \qquad t^{n}_j=
\left[ \begin{array}{ccccc}

\alpha^{n}_{11} & \alpha^{n}_{12}  & \cdots & \alpha^{n}_{1j} \\
 0 & \alpha^{n}_{22} & \cdots  & \alpha^{n}_{2j} \\
  0 & 0 &  &  &  \\
    \vdots & \vdots & \ddots  &\vdots  \\
     0 & 0 & \cdots  & \alpha^{n}_{jj} \\
\end{array} \right]$$
such that
$$(\mathbf{F}(f^\bullet))^{n}_j=\left\lbrace \begin{array}{lll}
(\lambda^{n-1}_{P^\bullet})_j t^{n}_j+t^{n+1}_j (\mu^{n}_{Q^\bullet})_j & \text{if} \,\, n \,\, \text{is even} \\

\\
(\mu^{n-1}_{P^\bullet})_j t^{n}_j+t^{n+1}_j (\lambda^{n}_{Q^\bullet})_j & \text{if} \,\, n \,\, \text{is odd}
\end{array} \right.$$

\vspace{2mm}
By setting $n$ equal to $i-1,i,i+1$ we have the following equations:
\begin{equation}
f^{i+j-1}=d^{i+j-2}\alpha_{jj}^i+\alpha_{jj}^{i+1}e_{\lbrace N-1 \rbrace}^{i+j-1}= d^{-N+j}_{\lbrace N-1\rbrace}\alpha_{(j+1)(j+1)}^{i-1}+\alpha_{(j+1)(j+1)}^{i}e^{j-1}\label{eq-02}
\end{equation}
for any $1\leq j \leq N-1$,
\begin{equation} 
\alpha_{xy}^i=\sum_{k=x-1}^y d^{i-N+x-1+k}_{\lbrace N-1-k\rbrace}\alpha_{(k+1)(y+1)}^{i-1}+\alpha_{x(y+1)}^{i}e^{y-1} \label{eq-03}
\end{equation}
and
\begin{equation} 
\alpha_{pq}^i=d^{i-p+1}\alpha_{(p-1)q}^{i}+\sum_{k=p-1}^q \alpha_{1k}^{i+1}e^{q-1}_{\lbrace N-1-q+k \rbrace} \label{eq-04}
\end{equation}
for any $1\leq x\leq y \leq N-2$ and $2\leq p \leq q \leq N-1$. 
We define homotopy maps $s^m$, $i\leq m \leq i+N-1$ as
$$s^m=\left\lbrace \begin{array}{lllll}
\alpha^{i+1}_{11} & \text{if} \,\, m=i+N-1,  \\

\\
\alpha^{i}_{1(N-1)} & \text{if} \,\, m=i+N-2, \\

\\
\alpha^{i-1}_{(m-i+2)(m-i+2)}+\sum_{k=1}^{N-2-m+i}d^{m-N} \alpha_{(m-i+1)(k+1)}^{i-1}e^i_{\lbrace k-1\rbrace} & \text{if} \, \, i \leq m \leq i+N-3.
\end{array} \right.$$



\vspace{1mm}
As $r$ varies in $\Z$, the numbers $i=Nr$ give us a collection of morphisms as above. So we can construct homotopy maps $(s^m)_{m\in \Z}$. If $i=Nr$, then it is easy to show that $f^i,f^{i+1},...,f^{i+N-1}$ and the homotopy maps $(s^m)_{i \leq m \leq i+2N-1}$ satisfy relation (\ref{eq-00}).
\\
Now let $\varphi^\bullet \in \Hom_{\K(\Proj \TT_{N-1}(R))}(\mathbf{F}(Q^\bullet),\mathbf{F}(P^\bullet))$. We want to find $f^\bullet\in \Hom_{\K_N(\Proj R)}(Q^\bullet,P^\bullet)$ such that $\mathbf{F}(f^\bullet)=\varphi^\bullet$. Suppose that $i=Nr$ and $\varphi^i=(\varphi_j^i)_{1\leq j \leq N-1}$ where  
$$
\varphi_j^i=
\left[ \begin{array}{ccccc}

\beta^i_{11} & \beta^i_{12}  & \cdots & \beta^i_{1j} \\
 0 & \beta^i_{22} & \cdots  & \beta^i_{2j} \\
  0 & 0 &  &  &  \\
    \vdots & \vdots & \ddots  &\vdots  \\
     0 & 0 & \cdots  & \beta^i_{jj} \\
\end{array} \right]
$$
We again consider the above diagram. Our goal is to construct $(f^m)_{i\leq m \leq i+N-1}$. 
By assumption we have $\varphi_j^i (\lambda_{Q^\bullet}^{i-1})_j=(\lambda_{P^\bullet}^{i-1})_j\varphi_j^{i-1}$ and $\varphi_j^{i+1} (\mu_{Q^\bullet}^i)_j=(\mu_{P^\bullet}^i)_j \varphi_j^i$ for any $1\leq j \leq N-1$. Hence we have the following equations:
\begin{equation}
 d^{i-1}\beta_{pq}^{i-1}- \beta_{(p+1)q}^{i-1}=- \beta_{p(q-1)}^{i}+\beta_{pq}^{i}e^{i+q-2}, \qquad 1\leq p \leq N-1, \quad p < q \leq N-1 \label{eq-05}
\end{equation}
\begin{equation}
d^{i+p-2} \beta_{pp}^{i-1}= \beta_{pp}^{i} e^{i+p-2}, \qquad 1\leq p \leq N-1 \label{eq-06}
\end{equation}
and
\begin{equation}
 \sum_{k=2}^{N} d_{\lbrace k-1 \rbrace}^{i+N-k}\beta_{(N-k+1)(N-1)}^{i} = \sum_{k=1}^{N-1} \beta_{1k}^{i+1}e_{\lbrace k \rbrace}^{i+N-2}. \label{eq-07}
\end{equation}
We define 
$$f^{i}= \sum_{k=2}^{N-1} \beta_{2k}^{i-1}e_{\lbrace k-2 \rbrace}^{i} + \beta_{1(N-1)}^{i}e_{\lbrace N-2 \rbrace}^{i}. $$
For $i+1 \leq j \leq i+N-2$ define
$$f^{j}= \sum_{k=2+j}^{N-1} \beta_{(2+j)k}^{j-2}e_{\lbrace k-(j+2) \rbrace}^{j-2} + \sum_{k=1}^{j+1} d_{\lbrace j+1-k\rbrace}^{j+k-2}\beta_{k(N-1)}^{j-1}e_{\lbrace N-j-2 \rbrace}^{j}, $$
and
 $$f^{i+N-1}=\sum_{k=1}^{N-1} \beta_{1k}^{i+1}e_{\lbrace k-1 \rbrace}^{i+N-1}.$$
By (\ref{eq-07}) we have $$d^{i+N-1}f^{i+N-1}=f^{i+N}e^{i+N-1}$$ It is not so hard to see that $\mathbf{F}(f^\bullet)\sim \varphi^\bullet. $
\end{proof}
\begin{lemma}
\label{lemma 014}
Let $(P^\bullet,d^\bullet)\in \K^{\ac}_{N}(\Proj R)$. The image of $P^\bullet$ under $\mathbf{F}$ is an exact complex in $\K(\Proj \TT_{N-1}(R))$.
\end{lemma}
\begin{proof}
Suppose $i=2r$ and $m=Nr$. The following diagram shows the image of $P^\bullet$ under $\mathbf{F}$ in degree $i$, $i+1$, $i+2$ and $i+3$. 
$$
\begin{small}
\xymatrix{ 
 P^m \ar[r] \ar[d]^{\mu_{1}^{i}} & P^m\oplus P^{m+1}\ar[r] \ar[d]^{\mu_{2}^{i}}  & \cdots \ar[r]  &  P^m\oplus \cdots \oplus P^{m+N-2} \ar[d]^{\mu_{N-1}^{i}}  \\
P^{m+N-1} \ar[r] \ar[d]^{\lambda_{1}^{i+1}} & P^{m+N-1}\oplus P^{m+N}\ar[r] \ar[d]^{\lambda_{2}^{i+1}} & \cdots \ar[r] &  P^{m+N-1}\oplus \cdots \oplus P^{m+2N-3} \ar[d]^{\lambda_{N-1}^{i+1}} \\
P^{m+N} \ar[r] \ar[d]^{\mu_{1}^{i+2}} & P^{m+N}\oplus P^{m+N+1}\ar[r]\ar[d]^{\mu_{2}^{i+2}} & \cdots \ar[r] &  P^{m+N}\oplus \cdots \oplus P^{m+2N-2}\ar[d]^{\mu_{N-1}^{i+2}}
\\
P^{m+2N-1} \ar[r]  & P^{m+2N-1}\oplus P^{m+2N}\ar[r] & \cdots \ar[r] &  P^{m+2N-1}\oplus \cdots \oplus P^{m+3N-3}
 }
 \end{small}
$$
We want to show that $\im \mu_{j}^{i}=\Ker \lambda_{j}^{i+1}$ for any $1 \leq j \leq N-1$. Clearly $\im \mu_{j}^{i}\subseteq\Ker \lambda_{j}^{i+1}$. Let $(x_1,x_2,...,x_j)\in \Ker \lambda_{j}^{i+1}$. It is easy to show that there exists $y_t\in P^{m+t-1}$ for all $1 \leq t \leq j$ such that 
$$ x_p= \sum_{k=1}^{j-p+1} d^{i+p-1+(k-1)}_{\lbrace N-k\rbrace}(y_{p+k-1}),$$ for all $1 \leq p \leq t$. Hence $(x_1,x_2,...,x_j)=\mu_{j}^{i}(y_1,y_2,...,y_j)$.
\\
Likewise suppose that $i=2r+1$ and $m=Nr$. We show that $\im \lambda_{j}^{i+1}=\Ker \mu_{j}^{i+2}$ for any $1 \leq j \leq N-1$. Clearly $\im \lambda_{j}^{i+1}\subseteq\Ker \mu_{j}^{i+2}$. Let $(x_1,x_2,...,x_j)\in \Ker \mu_{j}^{i+2}$. Tt is easy to show that there exists $y_{t}\in P^{m+N-2+t}$ for all $1 \leq t \leq j$ such that 
$$ x_{j}=d^{m+2N-3}(y_{j}),$$
and
$$x_{q}=-y_{q+1}+d^{m+N+q-2}(y_{q})$$ for any $1 \leq q \leq t$ and $q\neq j$. 
Hence $(x_1,x_2,...,x_j)=\lambda_{j}^{i+1}(y_1,y_2,...,y_j)$.
\end{proof}
\begin{lemma}
\label{lemma 14}
The functor $\mathbf F:\K_N(\Proj(R))\longrightarrow \K(\Proj \TT_{N-1}(R))$ is an exact triangulated functor.
\end{lemma}
\begin{proof}
Clearly $\mathbf{F}$ is an exact functor. Since $\mathbf{F}$ preserves direct sum it is enough to show that $\mathbf{F}$ transforms $D^{j}_{N}(P) $ to a projective object of $\C(\Proj \TT_{N-1}(R))$. By lemma \ref{lemma 014} $\mathbf{F}(D^{j}_{N}(P))$ is a bounded exact complex in $\C(\Proj \TT_{N-1}(R))$, Since $D^{j}_{N}(P) $ is an $N$-exact complex.  
Hence $\mathbf{F}(D^{j}_{N}(P))$ is a projective object in $\C(\Proj \TT_{N-1}(R))$.
Now we show that $$\mathbf{F}(\Sigma P^\bullet)\cong \mathbf{F}(P^\bullet)[1]$$
Let $P^\bullet$ be a complex in $\mathbb{K}_N(\Proj(R))$. For $i \equiv 0 (\mod 2)$ let $m= \frac{iN}{2}$. For $1 \leq j \leq N-1$ and $1 \leq k \leq j$ let $\alpha^i_{j,k} : (\Sigma P^\bullet)^{m+k-1} \to \oplus_{l=m+N-1}^{m+N-2+j} P^l$ be a morphism given as
$$
\alpha^i_{j,k} = \bordermatrix{
~ & ~ & ~ & ~ & \cr
1 & d^{m+k}_{\{N-k-1\}} & d^{m+k+1}_{\{N-k-2\}} & \cdots & d^{m+k+N-2}_{\{-k+1\}} \cr
2 & d^{m+k}_{\{N-k\}} & d^{m+k+1}_{\{N-k-1\}} & \cdots & d^{m+k+N-2}_{\{-k+2\}} \cr
\vdots & \vdots & \vdots & & \vdots \cr
k & d^{m+k}_{\{N-2\}} & d^{m+k+1}_{\{N-3\}} & \cdots & d^{m+k+N-2}_{\{0\}} \cr
k+1 & 0 & 0 & \cdots & 0 \cr
\vdots & \vdots & \vdots & & \vdots \cr
j & 0 & 0 & \cdots & 0 \cr
}
$$
These morphisms define a map $\alpha^i_j : F(\Sigma P^\bullet)^i_j \to (\mathbf{F}(P^\bullet)[1])^i_j$ as
$$
\alpha^i_j = \begin{bmatrix}
\alpha^i_{j,1} & \alpha^i_{j,2} & \cdots & \alpha^i_{j,j}
\end{bmatrix}.
$$
For $i \equiv 1 (\mod 2)$ let $m=\frac{(i+1)N}{2}$. For $1 \leq j \leq N-1$ and $1 \leq k \leq j$ let $\alpha^i_{j,k} : (\Sigma P^\bullet)^{m+k-1} \to \oplus_{l=m}^{m+j-1} P^l$ be a morphism given as
$$
\alpha^i_{j,k} = \bordermatrix{
~ & ~ & ~ & ~ & ~\cr
1 & 0 & 0 & \cdots & 0 \cr
\vdots & \vdots & \vdots & & 0 \cr
k-1 & 0 & 0 & \cdots & 0 \cr
k & 1 & 0 & \cdots & 0 \cr
k+1 & 0 & 0 & \cdots & 0 \cr
\vdots & \vdots & \vdots & & 0 \cr
j & 0 & 0 & \cdots & 0 \cr
}
$$
Likewise these morphisms give a map $\alpha^i_j : F(\Sigma P^\bullet)^i_j \to (\mathbf{F}(P^\bullet)[1])^i_j$ defined by
$$\alpha^i_j = \begin{bmatrix}
\alpha^i_{j,1} & \alpha^i_{j,2} & \cdots & \alpha^i_{j,j}
\end{bmatrix}
$$
In the other direction for $i \equiv 0(\mod 2)$, define $\beta^i_j :(\mathbf{F}(P^\bullet)[1])^i_j  \to F(\Sigma P^\bullet)^i_j$ as
$$\beta^i_j = \begin{bmatrix}
\beta^i_{j,1}\\
\beta^i_{j,2}\\
\vdots\\
\beta^i_{j,j}\\
\end{bmatrix}
$$
where,
$$
\beta^i_{j,k} = \bordermatrix{
~ & 1 & \cdots & k-1 & k & k+1 & \cdots & j \cr
~ & 0 & \cdots & 0 & 0 & 0 & \cdots & 0 \cr
~ & \vdots & & \vdots & \vdots & \vdots & & \vdots \cr
~ & \vdots & & \vdots & 0 & \vdots & & \vdots \cr
~ & 0 & \cdots & 0 & 1 & 0 & \cdots & 0
}
$$
For $i \equiv 1(\mod 2)$ define $\beta^i_j : (\mathbf{F}(P^\bullet)[1])^i_j  \to F(\Sigma P^\bullet)^i_j$ as
$$\beta^i_j = \begin{bmatrix}
\beta^i_{j,1}\\
\beta^i_{j,2}\\
\vdots\\
\beta^i_{j,j}\\
\end{bmatrix}
$$
where,
$$
\beta^i_{j,k} = \left[
\begin{array}{cc}
0_{j-k+1\times k-1} & I_{j-k+1}\\
0_{k-1\times k-1} & 0_{k-1 \times j-k+1}
\end{array}
\right],
$$
in which, $I_{j-k+1}$ is the identity matrix of order $j-k+1$, and the other three entries are zero matrices of given size.

It is not so hard to see that the composition $\alpha^i \circ \beta^i : (\mathbf{F}(P^\bullet)[1])^i \to (\mathbf{F}(P^\bullet)[1])^i$ is the identity morphism.
One can show that $\beta \circ \alpha - 1$ is null-homotopic where the homotopy maps are defined as follows:
For $i \equiv 0(\mod 2)$, $1 \leq j \leq N-1$ and $1 \leq k \leq j$ let
$$
\psi^i_{k} = \left[
\begin{array}{cc}
0_{k\times N-k-1} & 0_{k \times k}\\
-I_{N-k-1} & 0_{N-k-1 \times k}
\end{array}
\right]
$$
and define
$$
s^i_j = \left[
\begin{array}{cccc}
\psi^i_1 & \psi^i_2 & \cdots & \psi^i_j\\
0        & \psi^i_1 & \ddots & \psi^i_{j-1}\\
\vdots   & \ddots   &        & \vdots\\
0        & \cdots   & 0      & \psi^i_1 
\end{array}
\right]
$$
For  $i \equiv 1(\mod 2)$ and $1 \leq j \leq N-1$ define $s^i_j=0$.
It is quite tedious to show that the morphisms $\alpha:F(\Sigma P^\bullet) \to F(P^\bullet)[1]$ and $\beta:F(P^\bullet)[1] \to F(\Sigma P^\bullet)$ are natural in $P^\bullet$.
\end{proof}
Putting it all together, we have 
\begin{proposition}
\label{proposition 10}
The functor $\mathbf{F}:\K_N(\Proj R)\longrightarrow \K(\Proj \TT_{N-1}(R))$ is a fully faithful triangle functor. 
\end{proposition}
\begin{definition}
We call an $N$-complex $P^\bullet$, K-projective, if $P^\bullet \in \C_N(\Proj R)$ and $$\Hom_{\K_N(R)}(P^\bullet,Y^\bullet)=0$$ for all $Y^\bullet\in\K_{N}^{\ac}(R)$. We denote by $\K_{N}^{Proj}(\Proj R)$ the category of all K-projective $N$-complexes. Dually we define the triangulated full subcategory $\K_{N}^{Inj}(\Inj R)$ consisting of complexes $I^\bullet$ of injectives such that $\Hom_{\K_N(R)}(X^\bullet,I^\bullet)=0$ for all $X^\bullet\in\K_{N}^{\ac}(R)$
\end{definition}

\begin{theorem}
\label{theorem 10}
We have the following triangle equivalences:
$$ \K_{N}^{Proj}(\Proj R)\cong \D_{N}(R) \qquad , \qquad \K_{N}^{Inj}(\Inj R)\cong \D_{N}(R)  $$

\end{theorem}
\begin{proof}
See \cite[Theorem 2.22.]{IKM14}.
\end{proof}
\begin{definition}
For an additive category $\mathcal{A}$ with arbitrary coproducts, an object $C$ is called compact in $\mathcal{A}$ if the canonical morphism $\coprod_{i}\Hom_{\mathcal{A}}(C,X_i)\rightarrow \Hom_{\mathcal{A}}(C,\coprod_i X_i)$ is an isomorphism for any coproduct $\coprod_i X_i$ in $\mathcal{A}$. We denote by $\mathcal{A}^c$ the subcategory of $\mathcal{A}$ consisting of all compact objects.
\end{definition}
\begin{definition}
Let $\mathcal{T}$ be a triangulated category. A non-empty subcategory $\mathcal{S}$ of $\mathcal{T}$ is said to be thick if it is a triangulated subcategory of $\mathcal{T}$ that is closed under retracts. If, in addition, $\mathcal{S}$ is closed under all coproducts allowed in $\mathcal{T}$, then it is localizing; if it is closed under all products in $\mathcal{T}$ it is colocalizing. 
\end{definition}
The following remark gives us a better understanding of the objects in a thick subcategory, see \cite{Kra06}.
\begin{remark}
Let $\mathcal{S}$ be a class of objects of a triangulated category $\mathcal{T}$. Then 
\begin{itemize}
\item $\Thick (\mathcal{S})=\bigcup_{n\in \mathbb{N}}\langle \mathcal{S}\rangle_n$, where
\begin{itemize}

\item[$ -$ ]$\langle \mathcal{S}\rangle_1$ is the full subcategory of $\mathcal{T}$ containing $\mathcal{S}$ and closed under finite direct sums, direct summands and shifts.
\item[$ -$ ] For $n>1$, $\langle \mathcal{S}\rangle_n$ is the full subcategory of $\mathcal{T}$ consisting of all objects $S$ such that there is a distinguished triangle $Y\rightarrow X \rightarrow Z \rightsquigarrow$ in $\mathcal{T}$ with $Y\in \langle \mathcal{S}\rangle_i$, and $Z\in \langle \mathcal{S}\rangle_j$ such that $i,j<n$ and $S$ is a direct summand of shifting of $X$.
\end{itemize} 
\end{itemize}
\end{remark}
The following theorem has been proved by Iyama and et. al. in \cite{IKM14}. As a result of proposition \ref{proposition 10}, we present another proof for this theorem.
\begin{theorem}
\label{theorem 11}
For a ring $R$, we have the following triangle equivalence:
$$ \D_{N}(R)\cong \D(\TT_{N-1}(R))$$
\end{theorem}
\begin{proof}
By theorem \ref{theorem 10}  $\K_{N}^{Proj}(\Proj R)\cong \D_{N}(R)$ and $\K_{Proj}(\Proj \TT_{N-1}(R))\cong \D(\TT_{N-1}(R))$, and we have the following diagram:
$$\xymatrix{ \K_{N}(\Proj R) \ar[r]^{\mathbf{F}\quad} &  \K(\Proj \TT_{N-1}(R)) \\
 \K_{N}^{Proj}(\Proj R) \ar@{^{(}->}[u] & \K_{Proj}(\Proj \TT_{N-1}(R)) \ar@{^{(}->}[u] }$$ 
 In addition $\D(\TT_{N-1}(R))^c\cong K^{\bb}(\Projj \TT_{N-1}(R))$. 
For $1\leq i \leq N-1$ let $\mathcal{R}_i$ be the following projective representation of $A_{N-1}$ :
$$0\rightarrow 0 \rightarrow \cdots \rightarrow R \rightarrow R \rightarrow \cdots \rightarrow R$$ where $R$ start in $i$-th position with identity morphisms afterward. 
\\
We can show that $\K^{\bb}(\Projj \TT_{N-1}(R))=\Thick(\lbrace \mathcal{R}^{\bullet}_1,...,\mathcal{R}^{\bullet}_{N-1}\rbrace)$ whenever $\mathcal{R}^{\bullet}_i$ is a complex $\cdots\rightarrow 0 \rightarrow \mathcal{R}_i \rightarrow 0\rightarrow 0 \rightarrow \cdots$ concentrated in degree 0. Now we show that each $\mathcal{R}^{\bullet}_i$ belong to $\im \mathbf{F}$. Let $R^\bullet$ be a complex $\cdots\rightarrow 0 \rightarrow R \rightarrow 0\rightarrow 0 \rightarrow \cdots$ concentrated in degree 0. $\mathbf{F}(\Sigma(\Theta^{N-2}R^\bullet))=\mathcal{R}^{\bullet}_{N-1}$ and $\mathbf{F}(\Theta^{N-1}R^\bullet)=\mathcal{R}^{\bullet}_{1}$, hence $\mathcal{R}^{\bullet}_{N-1},\mathcal{R}^{\bullet}_{1} \in \im \mathbf{F}$. On the other hand $\mathbf{F}(\Theta^{N-3}R^\bullet)=\cdots\rightarrow 0\rightarrow \mathcal{R}_{N-1} \rightarrow \mathcal{R}_{N-2} \rightarrow 0 \rightarrow \cdots$ and therefore there exists a short exact sequence $0^\bullet \rightarrow \mathcal{R}^{\bullet}_{N-2} \rightarrow \mathbf{F}(\Theta^{N-3}R^\bullet) \rightarrow \mathcal{R}^{\bullet}_{N-1}[1]\rightarrow 0^\bullet$ with degree-wise split exact sequences. So there exists a triangle $\mathcal{R}^{\bullet}_{N-2} \rightarrow \mathbf{F}(\Theta^{N-3}R^\bullet) \rightarrow \mathcal{R}^{\bullet}_{N-1}[1] \rightsquigarrow$ in $\K(\Proj \TT_{N-1}(R))$, hence $\mathcal{R}^{\bullet}_{N-2} \in \im\mathbf{F}$, since $\mathbf{F}(\Theta^{N-3}R^\bullet),\mathcal{R}^{\bullet}_{N-1}[1] \in \im\mathbf{F}$. Similarly by induction we can say that $\mathcal{R}^{\bullet}_i \in \im\mathbf{F}$ for $2\leq i \leq N-3$. Hence $\Thick(\lbrace \mathcal{R}^{\bullet}_1,...,\mathcal{R}^{\bullet}_{N-1}\rbrace) \subseteq \im\mathbf{F} \subseteq \D(\TT_{N-1}(R))$. But $\im\mathbf{F}$ is closed under coproduct and contains compact objects, therefore the restriction of functor $\mathbf{F}$ to $\D_N(R)$ is dense hence $\D_{N}(R)\cong \D(\TT_{N-1}(R))$. 
\end{proof} 
According to the above theorem, we have a triangle equivalence $$\D^-_N(R)\cong \D^-(\TT_{N-1}(R))$$ thus $\xymatrix{F|_{\K_N^-(\Proj R)}:\K_N^-(\Proj R) \ar[r]^{\quad \cong} & \K^-(\Proj \TT_{N-1}(R))}$, see \cite[Corollaries 4.11, 2.17]{IKM14}. Moreover $F|_{\K_N^-(\Proj R)}$ induces some triangle equivalences between subcategories of $\K_N^-(\Proj R)$ and subcategories of $\K^-(\Proj \TT_{N-1}(R))$, see \cite[Corollary 4.15]{IKM14}. We summarize all of these equivalences in the following diagram. Note that the existence of the first row follows from proposition \ref{proposition 10}.  
$$\xymatrix{ \K_{N}(\Proj R) \ar@{^{(}->}[r]^{\mathbf{F}\quad} &  \K(\Proj \TT_{N-1}(R)) \\
 \K_{N}^{-}(\Proj R) \ar@{^{(}->}[u] \ar[r]^{\cong \qquad} & \K^{-}(\Proj \TT_{N-1}(R)) \ar@{^{(}->}[u] \\
 \K_{N}^{-,\bb}(\Proj R) \ar@{^{(}->}[u] \ar[r]^{\cong \qquad} & \K^{-,\bb}(\Proj \TT_{N-1}(R)) \ar@{^{(}->}[u] \\
 \K_{N}^{\bb}(\Proj R) \ar@{^{(}->}[u] \ar[r]^{\cong \qquad} & \K^{\bb}(\Proj \TT_{N-1}(R)) \ar@{^{(}->}[u] \\
 \K_{N}^{\bb}(\Projj R) \ar@{^{(}->}[u] \ar[r]^{\cong \qquad} & \K^{\bb}(\Projj \TT_{N-1}(R)) \ar@{^{(}->}[u]
 }
  $$
At the end of this section we show that the functor $\mathbf{F}$ is dense, hence there exist an triangle equivalence between $\K_N(\Proj R)$ and $\K(\Proj \TT_{N-1}(R))$.
Before we give the proof we need to introduce a another functor.
\\
Let $Q$ be the quiver of type $A_n$. Any projective representation $\mathcal{P}$ of $Q$ is of the form $\mathcal{P}=\oplus_{i=1}^{n} e_{\lambda}^i(P^i)$, where $P^i$ is the cokernel of split monomorphism $\mathcal{P}_{i-1}\rightarrow \mathcal{P}_{i}$. For any projective representation $\mathcal{P}=\oplus_{i=1}^{n} e_{\lambda}^i(P^i)$ of $Q$, set $\widehat{\mathcal{P}}=\oplus_{i=1}^{n} e_{\rho}^i(P^i)$. Clearly $\widehat{\mathcal{P}}$ is an object in the category $\Proj^{op}(A_n)$, where $\Proj^{op}(A_n)$ is the category of all representations by projective modules with split epimorphism maps.
\\
Now we define a functor$\,\,\, \widehat{}:\Proj A_n\rightarrow \Proj^{op}A_n$ such that any $\mathcal{P}\in \Proj A_n$ is mapped under $\,\widehat{}\,\,$ to $\widehat{\mathcal{P}}$, as defined above, and for any morphism $\varphi=(\varphi_t)_{1 \leq t \leq n}$ in $\Hom( e_{\lambda}^i(P^i),e_{\lambda}^j(P^j))$ define $\widehat{\varphi}$ as follows:
$$\widehat{\varphi}=\left\lbrace \begin{array}{ll}
0 & i<j \\
(\varphi_t)_{t} & i\geq j
\end{array} \right.$$
It is easy to check that $\, \widehat{}\, $ is in fact an equivalence of categories. We also see that for a finite quiver $Q$, $\, \widehat{}\, $ is an equivalence of categories, see \cite{AEHS11}. The functor $\, \widehat{} \, $ can be naturally extended to a functor $\K(\Proj A_n)\rightarrow \K(\Proj^{op} A_n)$ which we denote again by $\,  \widehat{} \,\,\, $. So for any $X^\bullet \in \K(\Proj A_n)$, let $\widehat{X^\bullet}$ be the complex with $\widehat{X_i}$ as its $i$-th term and $\widehat{d_i}$ as its $i$-th diffrential. The functor $\,\,  \widehat{}\,\, $ also is an equivalence of homotopy categories. 
\\
we also need the following description of compact objects in $\K_N(\Proj R)$. Neeman \cite{Nee08} showed that an object $X^\bullet$ of $\K(\Proj R)$ is compact if and only if it is isomorphic, in $\K(\Proj R)$, to a complex $\Y$ satisfying
\begin{itemize}
\item[$(i)$] $\Y$ is a complex of finitely generated projective modules.
\item[$(ii)$] $Y^i=0$ if $i<<0$.
\item[$(iii)$] $\HH^i(\Y^*)=0$ if $i<<0$, where $\Y^*=\Hom(Y^\bullet,R)$.
\end{itemize}
He also showed that when the compact objects generate the category $K(\Proj R)$. 
\begin{proposition}
If $R$ is a left coherent ring, then the category $\K(\Proj R)$ is compactly generated.
\end{proposition}
This idea enables us to prove our main theorem: 
\begin{theorem}
\label{theorem 02}
For a left coherent ring $R$, we have  triangle equivalence
$$\K_N(\Proj R)\cong \K(\Proj \TT_{N-1}(R)).$$
\end{theorem}
\begin{proof}
In view of proposition \ref{proposition 10}, the functor $\mathbf{F}$ is full and faithful. Now we show that $\mathbf{F}$ is dense. Let $\mathcal{T}$ be a triangle subcategory of $\K_N(\Proj R)$ such that for all $\T\in \mathcal{T}$ we have the following conditions
\begin{itemize}
\item[(1)] $\Tt \in \K_{N}^{+}(\Projj R)$;
\item[(2)] There exists an integer $n\in \Z$ such that for every $i<n$ and $1 \leq r \leq N-1$, $\HH^i_r\Tt^*=0$ where $\Tt^*$ denotes the induced complex $\Hom(\Tt,R)$.
\end{itemize}
Clearly the duality $\Hom(-,R):\Projj R\longrightarrow \Projj R^{op}$ induces an equivalence $\mathcal{T}^{op}\cong \K_N^{-,b}(\Projj R^{op})$ of triangulated categories. On the other hand, by restricting the functor $\mathbf{F}$ to $\mathcal{T}$, we have 
$$\xymatrix{ \K_{N}(\Proj R) \ar[r]^{\mathbf{F} \quad} &  \K(\Proj \TT_{N-1}(R)) \\
 \mathcal{T} \ar@{^{(}->}[u] \ar[r]^{\mathbf{F}|_{\mathcal{T}}\,\,\,\,\,\,\quad} & \K({\Projj \TT_{N-1}(R)})^c \ar@{^{(}->}[u] }$$
 We want to show that $\mathbf{F}|_{\mathcal{T}}$ is an equivalence. It is easy to check that we have the following commutative diagram
$$\xymatrix{ \mathcal{T} \ar[r]^{\mathbf{F}|_{\mathcal{T}}\,\,\,\,\,\,\quad} \ar[d]^{\Hom(-,R)}_{\cong} &  \K({\Projj \TT_{N-1}(R)})^c \\
 \K_N^{-,b}(\Projj R^{op}) \ar[r]^{\mathbf{F}^{op} \quad}_{\cong \quad}  & \K_N^{-,b}(\Projj \TT_{N-1}(R^{op})) \ar[u]^{\widecheck{} \,\,\, \circ \Hom(-,R)}_{\cong} }$$
where $\mathbf{F}^{op}$ is the functor $\mathbf{F}$ in construction \ref{construction 01} whenever we define it on $\C_N(\Proj R^{op})$ and $ \, \widecheck{} \,$ is the quasi-inverse of the functor  $\, \widehat{}\, $. Therefore $\mathbf{F}|_{\mathcal{T}}=\widecheck{} \,\, \circ \Hom(-,R)\circ \mathbf{F}^{op}\circ \Hom(-,R) $, hence it is an equivalence. It shows that $\K({\Projj \TT_{N-1}(R)})^c\subseteq \im \mathbf{F}$. On the other hand  $\im \mathbf{F}$ is closed under coproduct and contains compact objects, therefore $\im \mathbf{F}=\K(\Proj \TT_{N-1}(R))$.
\end{proof}
\begin{corollary}
If $R$ is a left coherent ring, then the category $\K_N(\Proj R)$ is compactly generated.
\end{corollary}
In a dual manner, in view of construction \ref{construction 01} and lemmas \ref{lemma 12}, \ref{lemma 13} and \ref{lemma 14} we can embed the category $\K_N(\Inj R)$ into the category $\K(\Inj \TT_{N-1}(R))$. Since the compact objects of $\K(\Inj\TT_{N-1}(R))$ are different from $\K(\Proj \TT_{N-1}(R))$, the proof of theorem \ref{theorem 02} dose not work. However, when $R$ is an artin algebra the embedding is dense.
\begin{proposition}
Let $\Lambda$ be an artin algebra. We have a triangle equivalence
$$\K_N(\Inj \Lambda)\cong \K(\Inj \TT_{N-1}(\Lambda)). $$
\end{proposition}  
\begin{proof}
Let $D$ denote the duality between right and left $\Lambda$-modules. The adjoint pair of functors $-\otimes_{\Lambda} D(\Lambda)$ and $\Hom_{\Lambda}(D(\Lambda),-)$ induces an equivalence between $\Proj \Lambda$ and $\Inj \Lambda$, which resticts to an equivalence between $\Projj \Lambda$ and $\Injj \Lambda$.
Therefore the adjoint pair induces an equivalence between $\K_N(\Proj \Lambda)$ and $\K_N(\Inj \Lambda)$. So if we denote the embedding $\K_N(\Inj R)\hookrightarrow\K(\Inj \TT_{N-1}(R))$ by $\mathbf{G}$, then we have the following diagram:
$$\xymatrix{ \K_{N}(\Proj \Lambda) \ar[r]^{\cong \quad} \ar[d]^{\cong } &  \K_N(\Inj \Lambda) \ar@{^{(}->}[d]^{\mathbf{G} \quad} \\
 K(\Proj \TT_{N-1}(\Lambda))\ar[r]^{\cong} & K(\Inj \TT_{N-1}(\Lambda)) }$$ Hence $\mathbf{G}$ is an equivalence of categories. 
\end{proof}
As a final remark we will discuss about $N$-dualizing complex.
\begin{remark}
Let $R$ be a commutative noetherian ring with a dualizing complex $D$ and $Q$ be a finite quiver. In \cite{AEHS11} they showed that Grothendieck duality is extendable to path algebra $RQ$. Therefore when $Q=A_{N-1}$, we have
$$\D^{\bb}(\mmod \TT_{N-1}(R))^{op}\cong \D^{\bb}(\mmod \TT_{N-1}(R)).$$
On the other hand by \cite{IKM14} we have
 $$\D_N^{\bb}(\mmod R)\cong \D^{\bb}(\mmod \TT_{N-1}(R)).$$ 
 Hence $$\D_N^{\bb}(\mmod R)^{op}\cong \D_N^{\bb}(\mmod R). $$
 So in case $R$ has a dualizing complex, there exists a Grothendieck duality for derived category of $N$-complexes and the question is "What could be the definition of an $N$-dualizing complex?".
\end{remark}
\section{N-totally acyclic complexes:}
\label{N-totally acyclic complexes}
Let $\mathcal{A}$ be an additive category. We say that a complex $X^\bullet$ in $\mathcal{A}$ is acyclic if the complex $\Hom_{\mathcal{A}}(A,X^\bullet)$ of abelian groups is acyclic for all $A\in \mathcal{A}$. If in addition $\Hom_{\mathcal{A}}(X^\bullet,A)$ is acyclic for all $A\in \mathcal{A}$, then $X^\bullet$ is called totally acyclic. Let $\C_{\tac}(\mathcal{A})$ denote the full subcategory of $\C(\mathcal{A})$ consisting of totally acyclic complexes. Note that these definitions are up to isomorphism in $\K(\mathcal{A})$. The full triangulated subcategory
of $\K(\mathcal{A})$ consisting of totally acyclic complexes, will be denoted by $\K_{\tac}(\mathcal{A})$. For instance
if $\mathcal{A}=\Proj R$ is the class of projective objects in $\Mod R$ then the object $X^\bullet$ of $\K_{\tac}(\Proj R)$ is an  exact complex such that $\Hom_{R}(X^\bullet,P)$ is acyclic for all $P\in \Proj R$. The objects of $\K_{\tac}(\Proj R)$ will be called totally acyclic complexes of projectives  
\begin{remark}
\label{remark 03}
It is easy to see that $X^\bullet\in \K_{\tac}(\Proj R)$ if and only if $\Hom_{\K(\Proj R)}(P^\bullet,X^\bullet)=0$ and $\Hom_{\K(\Proj R)}(X^\bullet,P^\bullet)=0$ for all $P^\bullet\in \K^{\bb}(\Proj R)$. 
\end{remark}
Let $R$ be a Noetherian ring. For the rest of this section we are only considering the category $\Projj R$, i.e. the category of finitely generated projective left $R$-modules. Given an integer $n$ and a complex
$$\xymatrix{\cdots \ar[r]^{{d}^{n-3}} & X^{n-2} \ar[r]^{{d}^{n-2}} & X^{n-1} \ar[r]^{{d}^{n-1}} & X^{n} \ar[r]^{{d}^{n}} & X^{n+1} \ar[r]^{{d}^{n+1}} & X^{n+2} \ar[r]^{{d}^{n+2}}  & \cdots }$$ 
in $\Mod R$, we denote its brutal truncation at degree n by
$$\xymatrix{ \beta_{\leq n}(X^\bullet): \cdots \ar[r]^{{d}^{n-3}} & X^{n-2} \ar[r]^{{d}^{n-2}} & X^{n-1} \ar[r]^{{d}^{n-1}} & X^{n} \ar[r] &0 \ar[r] & 0 \ar[r]  & \cdots }$$
If $X^\bullet \in \K_{\tac}(\Projj R)$, then $\beta_{\leq n}(X^\bullet)\in \K^{-,\bb}(\Projj R)$. The brutal truncation at degree zero induces a map from the category $\K_{\tac}(\Projj R)$ to the category $\K^{-,\bb}(\Projj R)$. However, this map is not a functor. Now consider the singularity category 
$$\D_{\sg}^{\bb}(R)=\K^{-,\bb}(\Projj R)/\K^{\bb}(\Projj R). $$
The brutal truncation induces a triangle functor 
$$\beta_{proj}:\K_{\tac}(\Projj R)\longrightarrow \D_{\sg}^{\bb}(R).$$
This functor is always full and faithful. If $R$ is either an Artin ring or commutative Noetherian local ring, then the functor $\beta_{proj}$ is dense if and only if $R$ is Gorenstein, see \cite{Buc87}, \cite{Hap91} and \cite{BJO14}.
\\
Motivated by the discussion above, in this section we want to introduce $N$-totally acyclic complexes and $N$-singularity category. We show that the restriction of functor $\mathbf{F}$ to this category is an equivalence. As a result of this equivalence, we show that there exists equivalences between $N$-singularity category of $\Mod R$ and usual singularity category of $\Mod \TT_{N-1}(R)$.
\\
It is easy to see that an $N$-complex $X^\bullet\in \K_N(\Projj R)$ is $N$-acyclic if and only if $$\Hom_{\K_N(\Projj R)}(P^\bullet,X^\bullet)=0$$ for all $P^\bullet\in \K^{\bb}(\Projj R)$.
\begin{definition}
An $N$-complex $X^\bullet\in \K_N(\Projj R)$ is called $N$-totally acyclic if and only if $\Hom_{\K_N(\Projj R)}(P^\bullet,X^\bullet)=0$ and $\Hom_{\K_N(\Projj R)}(X^\bullet,P^\bullet)=0$ for all $P^\bullet\in \K^{\bb}_N(\Projj R)$. We denote by $\K^{\tac}_N(\Projj R)$ the category of all $N$-totally acyclic complexes in $\Projj R$.
\end{definition}
Clearly, if $X^\bullet\in \K^{\tac}_N(\Projj R)$ then $X^\bullet$ is an $N$-acyclic complex. 
\begin{proposition}
For a left coherent ring $R$, we have the following triangle equivalences.
\begin{itemize}
\item[$(i)$] $\K^{\ac}_N(\Projj R)\cong \K_{\ac}(\Projj \TT_{N-1}(R)) $.
\item[$(ii)$] $\K^{\tac}_N(\Projj R)\cong \K_{\tac}(\Projj \TT_{N-1}(R)) $.
\end{itemize}

\end{proposition}
\begin{proof}
$(i)$ By theorem \ref{theorem 02}, the functor $\mathbf{F}$ induced an equivalence
 $$\xymatrix{ \K_{N}(\Proj R) \ar[r]^{\mathbf{F} \quad}_{\cong \quad} &  \K(\Proj \TT_{N-1}(R)) \\
 \K^{\ac}_{N}(\Projj R) \ar@{^{(}->}[u] & \K_{\ac}(\Projj \TT_{N-1}(R)) \ar@{^{(}->}[u] }$$
Let $P^\bullet$ be an object of $\K^{\ac}_N(\Projj R)$. 
Hence $\Hom_{\K_N(\Projj R)}(Q^\bullet,P^\bullet)=0$ for all $Q^\bullet\in \K^{\bb}_N(\Projj R)$. Now let $\mathcal{P^\bullet}\in \K^{\bb}(\Projj \TT_{N-1}(R))$. There exists an object $Q^\bullet\in \K^{\bb}_N(\Projj R)$ such that $\mathbf{F}(Q^\bullet)=\mathcal{P}^\bullet$. We have 
\begin{small}
$$\Hom_{\K(\Projj \TT_{N-1}(R))}(\mathcal{P}^\bullet,\mathbf{F}(P^\bullet))= \Hom_{\K(\Projj \TT_{N-1}(R))}(\mathbf{F}(Q^\bullet),\mathbf{F}(P^\bullet))\cong \Hom_{\K_N(\Projj R)}(Q^\bullet,P^\bullet)=0.$$
\end{small}
It shows that $\mathbf{F}(P^\bullet)\in\K_{\ac}(\Projj \TT_{N-1}(R))$. Hence the functor $\mathbf{F}$ sends any object of the subcategory $\K^{\ac}_{N}(\Projj R)$ of $\K_N(\Proj R)$ to an object of the subcategory $\K_{\ac}(\Projj \TT_{N-1}(R))$ of $\K(\Proj \TT_{N-1}(R))$.
\\
Let $\mathcal{P}^\bullet\in \K_{\ac}(\Projj \TT_{N-1}(R))$. In a similar way there exists $P^\bullet\in \K^{\ac}_{N}(\Projj R)$ such that $\mathbf{F}(P^\bullet)=\mathcal{P}^\bullet$, Hence $\mathbf{F}|_{\K^{\ac}_{N}(\Projj R)}$ is dense.
\\
$(ii)$ It is similar to $(i)$ by use of remark \ref{remark 03}. 
\end{proof}
We define the $N$-singularity category $\D^{\sg}_N(R)$ of $R$ as a Verdier quotient 
$$\K_N^{-,\bb}(\Projj R)/\K_N^{\bb}(\Projj R).$$ 
\begin{remark}
\label{remark 44}
The brutal truncation at degree zero induces a triangle functor 
$$\K_N^{\tac}(\Projj R) \longrightarrow \D_N^{\sg}(R),$$ and the following diagram shows that this functor is always full and faithful. Moreover it
is a triangle equivalence of categories when $R$ is a Gorenstein ring.
$$\xymatrix{ \K_{N}^{\tac}(\Projj R) \ar[d]^{\mathbf{F} \quad}_{\cong} \ar[r] & \D_N^{\sg}(R)=\K_N^{-,\bb}(\Projj R)/\K_N^{\bb}(\Projj R) \ar[d]^{\widetilde{\mathbf{F}} \quad}_{\cong} \\
 \K_{\tac}(\Projj \TT_{N-1}(R)) \ar@{^{(}->}[r]^{\beta_{proj}\qquad \qquad \qquad} & \D^{\bb}_{\sg}(\TT_{N-1}(R))=\K^{-,\bb}(\Projj \TT_{N-1}(R))/\K^{\bb}(\Projj R) }$$
The functor $\widetilde{\mathbf{F}}$ is induced from equivalences in diagram after theorem \ref{theorem 11}.
\end{remark}
Remark \ref{remark 44} provides us with another interpretation of quotient category $$\K^{\infty,\bb}(\Projj R)/\K^{\bb}(\Projj R)$$
where $\K^{\infty,b}(\Projj R)$ is the homotopy category of unbounded complexes with bounded homologies. Iyama and et. al. showed that there is a triangle equivalence between the above quotient category and $\D^{\bb}_{\sg}(\TT_2(R))$, whenever $R$ is a Gorenstein ring, see \cite{IKM11}. Hence by remark \ref{remark 44} we have the following equivalence of categories
$$\K^{\infty,\bb}(\Projj R)/\K^{\bb}(\Projj R)\cong \D_3^{\sg}(R).$$

\end{document}